\documentclass[12pt, reqno]{amsart}
\usepackage{amsmath, amsthm, amscd, amsfonts, amssymb, graphicx, color}
\usepackage[bookmarksnumbered, colorlinks, plainpages]{hyperref}
\hypersetup{colorlinks=true,linkcolor=red, anchorcolor=green, citecolor=cyan, urlcolor=red, filecolor=magenta, pdftoolbar=true}

\newtheorem{theorem}{Theorem}[section]
\newtheorem{lemma}[theorem]{Lemma}

\newtheorem{corollary}[theorem]{Corollary}
\theoremstyle{definition}

\theoremstyle{remark}
\newtheorem{remark}[theorem]{Remark}

\numberwithin{equation}{section}

\begin{document}

\setcounter{page}{1}

\title[Some inequalities for $P$-class functions]{Some inequalities for $P$-class functions}

\author[I. Nikoufar, D. Saeedi]{Ismail Nikoufar$^1$$^{*}$ and Davuod Saeedi$^2$}

\address{$^{1}$Department of Mathematics, Payame Noor University, P.O. Box 19395-3697 Tehran, Iran.}
\email{\textcolor[rgb]{0.00,0.00,0.84}{nikoufar@pnu.ac.ir}}

\address{$^{2}$Department of Mathematics, Payame Noor University, P.O. Box 19395-3697 Tehran, Iran.}
\email{\textcolor[rgb]{0.00,0.00,0.84}{dsaeedi3961@gmail.com}}





\date{\today
\newline \indent $^{*}$Corresponding author}

\begin{abstract}
In this paper, we provide some inequalities for
$P$-class functions and self-adjoint operators on a Hilbert space including
an operator version of the Jensen's inequality and
the Hermite-Hadamard's type inequality.
We improve the H\"{o}lder-MacCarthy inequality by providing an upper bound.
Some refinements of 
the Jensen type inequality for $P$-class functions will be of interest.
\end{abstract}

\maketitle

\section{Introduction and Preliminaries}
\noindent

%
%
%

Let $\mathcal{H}$ be a Hilbert space and $B(\mathcal{H})$ be the algebra of all bounded linear operators on
$\mathcal{H}$. We say that an operator $A$ in $B(\mathcal{H})$ is positive and write $A \geq 0$ if $\langle Ax, x\rangle \geq 0$
for all $x\in\mathcal{H}$. The spectrum of an operator $A\in B(\mathcal{H})$ is denoted by $Sp(A)$.
A function $f:\Bbb{I}\to\Bbb{R}$ is a $P$-class function on $\Bbb{I}$ if
\begin{equation}\label{pi-eq}
f(\lambda x+(1-\lambda)y)\leq f(x)+f(y),
\end{equation}
where $x,y\in\Bbb{I}$ and $\lambda\in [0,1]$.
Some properties of $P$-class functions can be found in \cite{Dragomir-SJM1995, Dragomir-BAMS1998}.
The set of all $P$-class functions contains the set of all convex functions and the set of all nonnegative monotone functions.
Every non-zero $P$-class function is nonnegative valued.
Indeed, choose $\lambda=0$ and fix $y_0\in \Bbb{I}$ in \eqref{pi-eq}. Hence,
$$
f(y_0)\leq f(x)+f(y_0),
$$
where $x\in\Bbb{I}$. Thus, $f(x)\geq 0$ for all $x\in\Bbb{I}$.

Jensen's inequality for  convex  functions  is  one  of  the  most  important
result in the theory of inequalities due to the fact that many other famous
inequalities are particular cases of this for appropriate choices of the function
involved. Mond and Pe$\check{c}$ari$\acute{c}$ established an operator version of
the Jensen inequality for a convex function in \cite{Mond-HJM1993} (see also
\cite{Furuta-Zagreb2005}) as follows:

\begin{theorem}\label{thm0}
Let $f:[m, M] \to \Bbb{R}$ be a continuous convex function. If
$x \in \mathcal{H}$, $\langle x,x\rangle = 1$, then for every self-adjoint operator $C$ such that $mI\leq C \leq MI$,
\begin{equation}\label{eq0-thm1}
f(\langle Cx, x\rangle) \leq \langle f(C)x,x\rangle.
\end{equation}
for each $x\in \mathcal{H}$ with $\langle x,x\rangle = 1$.
\end{theorem}

As a special case of Theorem \ref{thm0} we have the following
H\"{o}lder-MacCarthy inequality.

\begin{theorem}\label{thm00}
\cite[Theorem 2]{Agarwal-Dragomir-CMA2010}
Let $C$ be a self-adjoint positive operator on a Hilbert space $\mathcal{H}$. Then
\begin{itemize}
\item[(i)] $\langle C^rx,x\rangle\geq \langle Cx,x \rangle^r$ for all $r>1$ and $x\in \mathcal{H}$ with $\langle x,x\rangle = 1$;
\item[(ii)] $\langle C^rx,x\rangle\leq \langle Cx,x \rangle^r$ for all $0<r<1$ and $x\in \mathcal{H}$ with $\langle x,x\rangle = 1$;
\item[(i)] If $C$ is invertible, then $\langle C^rx,x\rangle\geq \langle Cx,x \rangle^r$ for all $r<0$ and $x\in \mathcal{H}$ with $\langle x,x\rangle = 1$.
\end{itemize}
\end{theorem}

In this paper, we show that many general inequalities can be given for
$P$-class functions and self-adjoint operators on a Hilbert space including
an operator version of the Jensen's inequality and
the Hermite-Hadamard's type inequality for $P$-class functions.
We improve the H\"{o}lder-MacCarthy inequality by providing an upper bound.

%

\section{Mond and Pe$\check{c}$ari$\acute{c}$ inequality for $P$-class functions and its application}
Taking into account Theorem \ref{thm0} and its applications for various concrete
examples of convex functions, it is therefore natural to investigate the corresponding
results for the case of $P$-class functions and its special cases.

\begin{theorem} \label{thm1}
Let $C$ be a self-adjoint operator on
the Hilbert space $\mathcal{H}$ and assume that $Sp(C)\subseteq [m,M]$ for some scalars $m,M$ with
$m < M$. If $f$ is a continuous $P$-class function on $[m, M]$, then
\begin{equation}\label{eq0-thm1}
f(\langle Cx, x\rangle) \leq 2\langle f(C)x,x\rangle
\end{equation}
for each $x\in \mathcal{H}$ with $||x||=1$. 
\end{theorem}
\begin{proof}
Since $f$ is $P$-class,
\begin{equation}\label{eq1-thm1}
f(\lambda x+(1-\lambda)y)-f(y)\leq f(x)
\end{equation}
for every $x,y\in[m,M]$, and $\lambda\in (0,1)$.
Consider
\begin{equation}\label{eq2-thm1}
\alpha:=\min_{y\in[m,M]}\frac{f(\lambda x+(1-\lambda)y)-f(y)}{\lambda(x-y)}.
\end{equation}
It follows from \eqref{eq1-thm1} that $\alpha\lambda(x-y)\leq f(x)$ and so $\alpha(x-y)\leq \frac{1}{\lambda}f(x)$.
Notice that $l(x):=\alpha(x-y)$ is a linear equation and $l(x)\leq \frac{1}{\lambda}f(x)$ for every $x\in[m,M]$.
By assumption, $m\leq \bar{g}\leq M$ where $\bar{g}=\langle Cx,x\rangle$.
Consider the straight line $l'(x):=\alpha(x-\bar{g})+f(\bar{g})$ passing through the point
$(\bar{g},f(\bar{g}))$ and parallel to the line $l$.
By continuity of $f$, we get
\begin{equation}\label{eq3-thm1}
l'(\bar{g})\geq f(\bar{g})-\epsilon
\end{equation}
for arbitrary $\epsilon > 0$.
We realize two cases:

(i) Let $l'(x)\leq \frac{1}{\lambda}f(x)$ for every $x\in[m,M]$.
Then, $l'(C)\leq \frac{1}{\lambda}f(C)$.  Hence,
\begin{equation}\label{eq4-thm1}
\langle l'(C)x,x \rangle\leq \frac{1}{\lambda}\langle f(C)x,x \rangle.
\end{equation}
By using \eqref{eq3-thm1}, \eqref{eq4-thm1} and linearity of $l'$, we observe that
\begin{align*}
f(\langle Cx,x\rangle)-\epsilon\leq l'(\langle Cx,x\rangle)=\langle l'(C)x,x\rangle\leq \frac{1}{\lambda}\langle f(C)x,x \rangle.
\end{align*}
Since $\epsilon$ is arbitrary, we deduce
\begin{equation}\label{eq5-thm1}
f(\langle Cx,x\rangle)\leq \frac{1}{\lambda}\langle f(C)x,x \rangle.
\end{equation}

(ii) There exits some points $x\in[m,M]$ such that $l'(x)>\frac{1}{\lambda}f(x)$.
Let
$$
A:=\{x\in [m,\bar{g}]: l'(x)>\frac{1}{\lambda}f(x)\},
$$
$$
B:=\{x\in [\bar{g},M]: l'(x)>\frac{1}{\lambda}f(x)\}.
$$
Consider $x_A:=\max\{x: x\in A\}$ and $x_B:=\min\{x: x\in B\}$.
Let $l_A$ be the line passing through the points $(x_A,0)$ and $(\bar{g},f(\bar{g}))$
and $l_B$ the line passing through the points $(x_B,0)$ and $(\bar{g},f(\bar{g}))$.
Define
%
$$
L(x):=\Big\{
\begin{array}{ll}
l_A(x), x\in [m,\bar{g}],\\
l_B(x), x\in [\bar{g},M].
\end{array}
$$
We show that $L(x)\leq \frac{1}{\lambda}f(x)$ for every $x\in[m, M]$.
We consider the partition $\{m, x_A, \bar{g}, x_B, M\}$ for the closed interval $[m, M]$.
Note that $l_A(x)\leq 0$ for every $x\in[m, x_A]$ and since $f(x)\geq 0$, we reach
$l_A(x)\leq \frac{1}{\lambda}f(x)$ for every $x\in[m, x_A]$.
On the other hand, one clearly has
\begin{equation}\label{main1}
l'(x)\leq \frac{1}{\lambda}f(x)
\end{equation}
for every $x\in (x_A, \bar{g}]$, otherwise,
there exists $x_0\in (x_A, \bar{g}]$ such that $l'(x_0)> \frac{1}{\lambda}f(x_0)$.
This infers $x_0\in A$ and so $x_0< x_A$, which is a contradiction.
So, by letting $x$ tends to $x_A$ from right in \eqref{main1}, one can deduce $l'(x_A)\leq \frac{1}{\lambda}f(x_A)$.
Moreover, since $x_A\in \bar{A}$, $l'(x_A)\geq \frac{1}{\lambda}f(x_A)$ and hence $l'(x_A)=\frac{1}{\lambda}f(x_A)$.
It follows that $l'$ is the line passing through the points $(x_A,\frac{1}{\lambda}f(x_A))$ and $(\bar{g},f(\bar{g}))$
and the slope of $l'$ is $\alpha=\frac{f(\bar{g})-\frac{1}{\lambda}f(x_A)}{\bar{g}-x_A}$, where
the slope of $l_A$ is $\alpha'=\frac{f(\bar{g})}{\bar{g}-x_A}$.
By the inequality \eqref{main1} we have
$$
l_A(x)=\alpha'(x-\bar{g})+f(\bar{g})\leq \alpha(x-\bar{g})+f(\bar{g})=l'(x)\leq \frac{1}{\lambda}f(x)
$$
for every $x\in (x_A, \bar{g}]$.
So, $L(x)=l_A(x)\leq\frac{1}{\lambda}f(x)$ for every $x\in [m, \bar{g}]$.

By the same way, one has $L(x)=l_B(x)\leq\frac{1}{\lambda}f(x)$ for every $x\in [\bar{g}, M]$.
Note that $l_A(\bar{g})=l_B(\bar{g})$ and since $f$ is continuous,
\begin{equation}
l_A(\bar{g})\geq f(\bar{g})-\epsilon
\end{equation}
for arbitrary $\epsilon > 0$. For the case where $\textrm{sp}(C)\subseteq [m,\bar{g}]$,
\begin{align*}
f(\langle Cx,x\rangle)-\epsilon\leq l_A(\langle Cx,x\rangle)=\langle l_A(C)x,x\rangle\leq \frac{1}{\lambda}\langle f(C)x,x \rangle.
\end{align*}
Moreover, when $\textrm{sp}(C)\subseteq [\bar{g},M]$, we have
\begin{align*}
f(\langle Cx,x\rangle)-\epsilon\leq l_A(\langle Cx,x\rangle)=l_B(\langle Cx,x\rangle)=\langle l_B(C)x,x\rangle\leq \frac{1}{\lambda}\langle f(C)x,x \rangle
\end{align*}
and so we obtain \eqref{eq5-thm1}. According to \eqref{eq5-thm1} and for $\lambda=\frac{1}{2}$ we deduce \eqref{eq0-thm1}.
We claim that $\frac{1}{2}$ is the best possible for $\lambda$ in \eqref{eq5-thm1}.

(1) Let $0<\lambda\leq\frac{1}{2}$. So, $\frac{1}{\lambda}\geq2$ and consequently by \eqref{eq0-thm1}, we deduce
\begin{equation}\label{eq6-thm1}
f(\langle Cx, x\rangle) \leq 2\langle f(C)x,x\rangle < \frac{1}{\lambda}\langle f(C)x,x\rangle.
\end{equation}

(2) Let $\frac{1}{2}<\lambda<1$ and note that
the function $g(t)=\frac{2-t^2}{\alpha}$, $t\in[-1,1]$, is a $P$-class function for every $\alpha\geq 1$.
Consider
$
C=\Big[
\begin{array}{ll}
-1 & 0\\
0 & 1
\end{array}
 \Big]
$
and $x=(\frac{1}{\sqrt{2}},\frac{1}{\sqrt{2}})$.
Then, $g(\langle Cx, x\rangle)=g(0)=\frac{2}{\alpha}$ and $\langle g(C)x,x\rangle=\frac{1}{\alpha}$.
Since $g$ is $P$-class, by \eqref{eq5-thm1}, we have
$g(\langle Cx, x\rangle) \leq \frac{1}{\lambda}\langle g(C)x,x\rangle$ and so $\lambda\leq\frac{1}{2}$
which is a contradiction.
\end{proof}


\begin{corollary} \label{cor1}
Under the hypotheses of Theorem \ref{thm1},
if $x \in \mathcal{H}$, $||x||\neq 1$, then 
\begin{equation}\label{eq0-cor1}
f\Big(\frac{\langle Cx, x\rangle}{\langle x, x\rangle}\Big) \leq \frac{2\langle f(C)x,x\rangle}{\langle x, x\rangle}.
\end{equation}
\end{corollary}
\begin{proof}
Let $y:=\frac{x}{\sqrt{\langle x, x\rangle}}$ and apply Theorem \ref{thm1}.
\end{proof}
%
\begin{lemma}\label{lemm1}
Let $f$ be a continuous $P$-class function and $\lambda<0$.
If $f$ is decreasing, then
\begin{equation}\label{eq-m1-lem1}
f((1-\lambda)x+\lambda y)\geq f(x)-f(y)
\end{equation}
for every $x,y\in [m,M]$ with $x<y$.
\end{lemma}
\begin{proof}
We have $(1-\lambda)x+\lambda y=x+\lambda(y-x)\leq x$.
Since $f$ is decreasing,
$$
f((1-\lambda)x+\lambda y)\geq f(x)\geq f(x)-f(y).
$$

\end{proof}

%
\begin{lemma}\label{lemm2}
Let $f$ be a continuous $P$-class function and $\lambda>1$.
If $f$ is increasing, then \eqref{eq-m1-lem1} holds,
\end{lemma}
\begin{proof}
We have $(1-\lambda)x>(1-\lambda)y$ and so $(1-\lambda)x+\lambda y \geq y$.
Since $f$ is increasing, we obtain
$$
f((1-\lambda)x+\lambda y)\geq f(y)\geq f(x)\geq f(x)-f(y).
$$

\end{proof}

%

\begin{theorem}\label{thm2}
Let $f:[m, M]\to\Bbb{R}$ be a continuous decreasing $P$-class function and let
the self-adjoint operator $C$ satisfies $mI\leq C \leq MI$.
If $0 < \langle x,x \rangle < u$, $x\in\mathcal{H}$,  
$a\in [m, M]$, and $\frac{ua - \langle Cx, x\rangle}{u - \langle x,x\rangle}\in
[m, M]$, then
\begin{equation}\label{eq0-thm2}
f\Big(\frac{ua - \langle Cx, x\rangle}{u - \langle x,x\rangle}\Big)\geq f(a)-\frac{2\langle f(C)x, x\rangle}{\langle x,x\rangle}.
\end{equation}
\end{theorem}
\begin{proof}
Applying Lemma \ref{lemm1} with $\lambda=-\frac{\langle x,x\rangle}{u-\langle x,x\rangle}<0$, $x=a$, $y=\frac{\langle Cx, x\rangle}{\langle x,x\rangle}$, and Corollary \ref{cor1}, we find that
\begin{align}
f\Big(\frac{ua - \langle Cx, x\rangle}{u - \langle x,x\rangle}\Big)
&=f\Big(\frac{u}{u-\langle x,x\rangle}a -\frac{\langle x,x\rangle}{u-\langle x,x\rangle}\frac{\langle Cx, x\rangle}{\langle x,x\rangle}\Big)\nonumber\\
&\geq f(a)-f\Big(\frac{\langle Cx, x\rangle}{\langle x,x\rangle}\Big)\nonumber\\
&\geq f(a)-\frac{2\langle f(C)x, x\rangle}{\langle x,x\rangle}.
\end{align}
\end{proof}

%
\begin{corollary}
Under the hypotheses of Theorem \ref{thm2}, if $f$ is increasing, then 
\begin{equation}\label{eq0-cor2}
f\Big(\frac{ua - \langle Cx, x\rangle}{u - \langle x,x\rangle}\Big)\geq
f\Big(\frac{\langle Cx, x\rangle}{\langle x,x\rangle}\Big)-f(a).
\end{equation}
\end{corollary}
\begin{proof}
Applying Lemma \ref{lemm2} with $\lambda=\frac{u}{u-\langle x,x\rangle}>1$, $x=\frac{\langle Cx, x\rangle}{\langle x,x\rangle}$, $y=a$, and Corollary \ref{cor1}, we obtain the result.
\end{proof}

%
\begin{theorem}\label{thm3}
Let the conditions of Theorem \ref{thm1} be satisfied. Then
\begin{equation}\label{eq0-thm3}
\langle f(C)x, x\rangle\leq f(m)+f(M).
\end{equation}
\end{theorem}
\begin{proof}
Let $u\in[m,M]$. Then $u=\frac{M-u}{M-m}m+\frac{u-m}{M-m}M$.
The function $f$ is $P$-class, so $f(u)\leq f(m)+f(M)$.
The operator $f(m)+f(M)-f(C)$ is positive, and hence, \eqref{eq0-thm3} follows.
\end{proof}

%
\begin{theorem}\label{thm4}
Let the conditions of Theorem \ref{thm1} be satisfied.
Let $J$ be an interval such that $f([m, M])\subset J$.
If $F(u, v)$ is a real function defined on $J\times J$ and non--decreasing in $u$, then
\begin{align}\label{eq0-thm4}
F(2\langle f(C)x,x\rangle,f(\langle Cx,x\rangle))&\leq\max_{t\in[m,M]}F(2(f(m)+f(M)),f(t))\nonumber\\
&=\max_{\theta\in[0,1]}F(2(f(m)+f(M)),f(\theta m+(1-\theta)M)).
\end{align}
\end{theorem}
\begin{proof}
According to the non-decreasing character of $F$ and Theorem \ref{thm3}, we deduce
\begin{align*}
F(2\langle f(C)x,x\rangle,f(\langle Cx,x\rangle))
&\leq F(2(f(m)+f(M)),f(\bar{g}))\nonumber\\
&\leq \max_{t\in[m,M]}F(2(f(m)+f(M)),f(t))
\end{align*}
since $\bar{g} = \langle Cx,x\rangle\in [m, M]$.
The second form of the right side of \eqref{eq0-thm4} follows
at once from the change of variable $\theta=\frac{M-t}{M-m}$, so that
$t=\theta m+(1-\theta)M$, with $0\leq \theta\leq 1$.
\end{proof}

In the same way (or more simply just by replacing $F$ by $-F$ in the
above theorem) we can prove the following:
%
\begin{corollary}\label{cor3}
Under the same hypotheses as Theorem \ref{thm4}, except that $F$ is
non--increasing in its first variable, we have
\begin{align*}
F(2\langle f(C)x,x\rangle,f(\langle Cx,x\rangle))
&\geq\min_{t\in[m,M]}F(2(f(m)+f(M),f(t)))\nonumber\\
&=\min_{\theta\in[0,1]}F(2(f(m)+f(M)),f(\theta m+(1-\theta)M)).
\end{align*}
\end{corollary}

\begin{corollary} \label{cor4}
Let the conditions of Theorem \ref{thm1} be satisfied. Then,
\begin{itemize} 
\item[(i)] $2\langle f(C)x,x\rangle\leq \lambda f(\langle Cx, x\rangle)$ for some $\lambda>0$,
\item[(ii)] 
$2\langle f(C)x,x\rangle\leq \lambda +f(\langle Cx, x\rangle)$ for some $\lambda\in\Bbb{R}$.
\end{itemize}
\end{corollary}
\begin{proof}
(i) Consider $F(u,v)=\frac{u}{v}$, $\varphi(t)=\frac{2(f(m)+f(M))}{f(t)}$, and $J=(0,\infty)$.
So, $F$ is non-decreasing on its first variable and by Theorem \ref{thm4} we have
$$
\frac{2\langle f(C)x,x\rangle}{f(\langle Cx,x\rangle)}\leq \max_{t\in[m,M]}\varphi(t)=\frac{2(f(m)+f(M))}{\min_{t\in[m,M]}f(t)}.
$$
The function $\varphi$ essentially attains its maximum value when the function $f$
attains its minimum value on $[m,M]$ by continuity of $f$.
Hence, by letting  $\lambda=\frac{2(f(m)+f(M))}{\min_{t\in[m,M]}f(t)}$, we find the result.

(ii) Consider $F(u,v)=u-v$, $\varphi(t)=2(f(m)+f(M))-f(t)$, and $J=\Bbb{R}$.
So, $F$ is non-decreasing on its first variable and Theorem \ref{thm4} leads
$$
2\langle f(C)x,x\rangle-f(\langle Cx,x\rangle)\leq \max_{t\in[m,M]}\varphi(t)=2(f(m)+f(M))-\min_{t\in[m,M]}f(t).
$$
The function $f$
attains its minimum value by continuity of $f$.
Hence, it suffices to let $\lambda=2(f(m)+f(M))-\min_{t\in[m,M]}f(t)$.
\end{proof}



Combining Theorem \ref{thm1} and Corollary \ref{cor4} we identify the following result.

\begin{corollary} 
Let the conditions of Theorem \ref{thm1} be satisfied. Then,
\begin{itemize}
\item[(i)] $\frac{2}{\lambda}\langle f(C)x,x\rangle\leq f(\langle Cx, x\rangle)\leq 2\langle f(C)x,x\rangle$ for some $\lambda>0$,
\item[(ii)] $0\leq 2\langle f(C)x,x\rangle-f(\langle Cx, x\rangle) \leq \lambda$ for some $\lambda\in\Bbb{R}$.
\end{itemize}
\end{corollary}
For instance, when $f(t)=t^r$, $0<r<1$ and $t\in[m,M]$, we obtain
$$
0\leq 2\langle C^rx,x\rangle - \langle Cx,x\rangle^r\leq 2M^r+m^r
$$
and when $f(t)=\ln t$, $t\in[m,M]\subseteq [1,\infty)$, $f$ is $P$-class and we have
$$
\frac{\ln m}{\ln M+\ln m}\langle \ln(C)x,x\rangle\leq \ln(\langle Cx, x\rangle)\leq 2\langle \ln(C)x,x\rangle,
$$
$$
0\leq 2\langle \ln(C)x,x\rangle - \ln(\langle Cx, x\rangle)\leq 2\ln(M)+\ln(m).
$$

As a consequence of the definition of a $P$-class function one can verify that if $f$ is a continuous increasing
$P$-class function and $g$ is a convex function, then $f\circ g$ is a $P$-class function.
Remember that $f$ is homogeneous, whenever,
$f(\lambda A) = \lambda f(A)$ for $\lambda > 0$.
We have the following simple corollary.
%
\begin{corollary}
Let the conditions of Theorem \ref{thm1} be satisfied and let
$f$ be a non-decreasing function and $n\geq 1$.
\begin{itemize}
\item[(i)] If $f$ is homogeneous, then $f^n(\langle Cx,x\rangle)\leq 2^n\langle f^n(C)x,x\rangle$.
\item[(ii)] If $f$ is subadditive, then $f^n$ is $P$-class and $f^n(\langle Cx,x\rangle)\leq 2\langle f^n(C)x,x\rangle$.
\end{itemize}
\end{corollary}

In the next corollary, we obtain the Hermite-Hadamard's type inequality for $P$-class functions.

%
\begin{corollary}\label{thm5}
Let the conditions of Theorem \ref{thm1} be satisfied and let $p$ and $q$
be nonnegative numbers, with $p + q > 0$, for which
\begin{align*}
\langle Cx,x \rangle=\frac{pm+qM}{p+q}.
\end{align*}
Then
\begin{align*}
\frac{1}{2}f\Big(\frac{pm+qM}{p+q}\Big)\leq \langle f(C)x,x \rangle\leq f(m)+f(M).
\end{align*}
\end{corollary}
\begin{proof}
By virtue of Theorem  \ref{thm1} and \ref{thm3} we reach
\begin{align*}
f\Big(\frac{pm+qM}{p+q}\Big)=f(\langle Cx,x \rangle)\leq 2\langle f(C)x,x \rangle)\leq 2(f(m)+f(M)).
\end{align*}
\end{proof}

We can improve the H\"{o}lder-MacCarthy inequality by providing an upper bound.
We use the fact that the function $t^r$, $0<r<1$, is $P$-class, in addition to being concave.
%
\begin{lemma}\label{lm-Mac}
Let $\alpha, \beta>0$ and $0<r<1$. Then, $(\alpha+\beta)^r\leq\alpha^r+\beta^r$.
\end{lemma}
\begin{proof}
Define $f_r(t)=(1+t)^r-t^r$, $t>0$ and note that
$f'_{r}(t)<0$. So, $f_{r}$ is decreasing and
the result follows from the fact that $f_r(\frac{\alpha}{\beta})\leq f_r(0)$.
\end{proof}
%
\begin{corollary}\label{Mac1}
Let $C$ be a self-adjoint positive operator on a Hilbert space $\mathcal H$. Then
\begin{itemize}
\item[(i)] for all $0<r<1$ and $x\in \mathcal H$ with $||x||=1$,
\begin{align}\label{eq-corMac1}
\langle C^rx,x\rangle\leq \langle Cx,x\rangle^r\leq 2\langle C^rx,x\rangle,
\end{align}
\item[(ii)] for all $r>1$ and $x\in \mathcal H$ with $||x||=1$,
\begin{align}\label{eq-corMac2}
\langle Cx,x\rangle^r\leq \langle C^rx,x\rangle\leq 2^r\langle Cx,x\rangle^r.
\end{align}
\end{itemize}
\end{corollary}
\begin{proof}
(i) The first inequality is H\"{o}lder-MacCarthy inequality for the case where $0<r<1$.
Let $0<a<b$ and $0<\lambda<1$. In view of Lemma \ref{lm-Mac} we get
$$
(\lambda a+(1-\lambda)b)^r\leq (\lambda a)^r+((1-\lambda)b)^r\leq a^r+b^r.
$$
This ensures the function $t^r$ is $P$-class and hence using Theorem \ref{thm1} we reach
the second inequality.

(ii) By applying $\frac{1}{r}<1$ in part (i) we have
\begin{align}\label{eq-corMac3}
\langle C^{1/r}x,x\rangle\leq \langle Cx,x\rangle^{1/r}\leq 2\langle C^{1/r}x,x\rangle.
\end{align}
Replacing $C^r$ with $C$ in \eqref{eq-corMac3} we deduce
\begin{align*}
\langle Cx,x\rangle\leq \langle C^rx,x\rangle^{1/r}\leq 2\langle Cx,x\rangle,
\end{align*}
which implies the result.
\end{proof}

Let $w_i$, $x_i$ be positive numbers with
$\sum_{i=1}^{n}w_i= 1$. Then the weighted power means are
defined by
$$
M_n^{[r]}(x;w)=\Big(\sum_{i=1}^{n}w_ix_i^r\Big)^{1/r}, \ \ r\neq 0
$$
and
$$
M_n^{[0]}(x;w)=\prod_{i=1}^{n}x_i^{w_i}
$$
is called weighted geometric mean and denoted by $G_w$.
It is well-known that if $s\leq r$, then
\begin{equation}\label{wpm-inc}
M_n^{[s]}(x;w)\leq M_n^{[r]}(x;w).
\end{equation}

The weighted arithmetic mean of a non-empty sequence of data $\{x_1, x_2, ..., x_n\}$
and corresponding non-negative weights $\{w_{1},w_{2}, ..., w_{n}\}$
with $\sum_{i=1}^{n}w_i= 1$
is defined by
$$
A_w=\sum_{i=1}^{n}w_ix_i
$$
and the weighted harmonic mean of them is defined by
$$
H_w=\Big(\sum_{i=1}^{n}w_ix_i^{-1}\Big)^{-1}.
$$
The arithmetic-geometric-harmonic mean inequality is a well-known inequality as follows:
$$
H_w\leq G_w \leq A_w.
$$
According to improved H\"{o}lder-MacCarthy inequality we identify the following
relation between the weighted arithmetic mean and the weighted power mean.

\begin{corollary}\label{HAG1}
Let $w_i$, $x_i$ be positive numbers with
$\sum_{i=1}^{n}w_i= 1$. Then
\begin{itemize}
\item[(i)] for all $0<r<1$,
$$
M_n^{[r]}(x;w)\leq A_w\leq 2^{1/r}M_n^{[r]}(x;w),
$$
\item[(ii)] for all $r>1$,
$$
A_w\leq M_n^{[r]}(x;w)\leq2A_w.
$$
\end{itemize}
\end{corollary}
\begin{proof}
(i) Consider
$$
C=\left(
\begin{array}{lll}
x_1 & \cdots & 0\\
\vdots & \ddots & \vdots\\
0 & \cdots & x_n
\end{array}
\right) \ \ \ \ \ \  \textrm{and} \ \ \ \ \ \
x=\left(
\begin{array}{l}
\sqrt{w_1}\\
\vdots\\
\sqrt{w_n}
\end{array}
\right).
$$
Clearly, we have $\langle Cx,x\rangle^r=(\sum_{i=1}^{n}w_ix_i)^r$
and
$\langle C^rx,x\rangle=\sum_{i=1}^{n}w_ix_i^r$.
In view of \eqref{eq-corMac1}, we obtain the desired result.

(ii) By considering $C$ and $x$ as above and applying \eqref{eq-corMac2} we get the result.
\end{proof}


Some refinements of the arithmetic-geometric-harmonic mean inequality are of interest.

\begin{remark}
Let $w_i$, $x_i$ be positive numbers with
$\sum_{i=1}^{n}w_i= 1$.

(i) For all $0<r<1$,
\begin{align*}
2^{-1/r}M_n^{[-r]}(x;w)&\leq H_w \leq M_n^{[-r]}(x;w)\\
&\leq G_w\leq M_n^{[r]}(x;w)\leq A_w\leq 2^{1/r}M_n^{[r]}(x;w).
\end{align*}
Replacing $x_i^{-1}$ with $x_i$ in Corollary \ref{HAG1}(i)
and applying the monotonically decreasing function $t^{-1}$ to both sides of the inequalities
we get the first and second inequalities. 
The third and forth inequalities obtain by \eqref{wpm-inc}.
We deduce the last two inequalities by Corollary \ref{HAG1}(i).

(ii) For all $r>1$,
$$
\frac{1}{2} H_w\leq M_n^{[-r]}(x;w)\leq H_w\leq G_w\leq A_w\leq M_n^{[r]}(x;w)\leq2A_w.
$$

Similar to that of part (i) and  Corollary \ref{HAG1}(ii)
we reach the first and second inequalities. 
The third and forth inequalities are well-known inequalities.
The last two inequalities are obtained in Corollary \ref{HAG1}(ii).
\end{remark}

%
%

\section{Multiple operator versions and its application}

%
%

In this section,
we investigate a multiple operator version of Theorem \ref{thm1} and
the corresponding applications for the $P$-class functions.

%
\begin{theorem}\label{M-Mond-pclass}
Let $C_i$ be self-adjoint operators with $Sp(C_i)\subseteq [m,M]$ for some scalars
$m<M$ and $x_i\in \mathcal H$, $i\in\{1,...,n\}$ with $\sum_{i=1}^{n}||x_i||^2=1$.
If $f$ is a $P$-class function on $[m,M]$, then
$$
f\Big(\sum_{i=1}^{n}\langle C_ix_i,x_i\rangle\Big)\leq 2\sum_{i=1}^{n}\langle f(C_i)x_i,x_i\rangle.
$$
\end{theorem}
\begin{proof}
We consider
$$
\tilde{C}=\left(
\begin{array}{lll}
C_1 & \cdots & 0\\
\vdots & \ddots & \vdots\\
0 & \cdots & C_n
\end{array}
\right) \ \ \ \ \ \  \textrm{and} \ \ \ \ \ \
\tilde{x}=\left(
\begin{array}{l}
x_1\\
\vdots\\
x_n
\end{array}
\right).
$$
By a simple verification we get $Sp(\tilde{C})\subseteq [m,M]$ and $||\tilde{x}||=1$.
On the other hands,
$$
f(\langle \tilde{C}\tilde{x},\tilde{x} \rangle)=f\Big(\sum_{i=1}^{n}\langle C_ix_i,x_i\rangle\Big),
$$
$$
\langle f(\tilde{C})\tilde{x}, \tilde{x}\rangle=\sum_{i=1}^{n}\langle f(C_i)x_i,x_i\rangle.
$$
According to Theorem \ref{thm1} we have
$f(\langle \tilde{C}\tilde{x},\tilde{x} \rangle) \leq 2\langle f(\tilde{C})\tilde{x}, \tilde{x}\rangle$
and so we deduce the desired result.
\end{proof}

The following particular case is of interest.
%
\begin{corollary}\label{M-Mond-pclass-cor}
Let $C_i$ be self-adjoint operators with $Sp(C_i)\subseteq [m,M]$, $i\in\{1,...,n\}$ for some scalars
$m<M$. 
If $f$ is a $P$-class function on $[m,M]$ and $p_i\geq 0$ with $\sum_{i=1}^{n}p_i=1$, then
$$
f\Big(\sum_{i=1}^{n}p_i\langle C_ix,x\rangle\Big)\leq 2\sum_{i=1}^{n}p_i\langle f(C_i)x,x\rangle
$$
for every $x\in \mathcal H$ with $||x||=1$.
\end{corollary}
\begin{proof}
It follows from Theorem \ref{M-Mond-pclass} by choosing $x_i=\sqrt{p_i}x$, $i\in\{1,...,n\}$, where $x\in\mathcal H$ with $||x||=1$.
\end{proof}

The following corollary is also of interest.

%
\begin{corollary}\label{M-Mond-sharp-pclass}
Let $f$ be a $P$-class function on $[m,M]$, $C_i$ self-adjoint operators with $Sp(C_i)\subseteq [m,M]$,
$i\in\{1,...,n\}$ and $p_i\geq 0$ with $\sum_{i=1}^{n}p_i=1$.
Assume that $I\subsetneq\{1,...,n\}$ and $I^{c}=\{1,...,n\}\backslash I$, $p_{I}=\sum_{i\in I}p_i$, $p_{I^{c}}=1-\sum_{i\in I}p_i$.
Then for any $x\in \mathcal H$ with $||x||=1$,
\begin{align*}
f\Big(\sum_{i=1}^{n}p_i\langle C_ix,x\rangle\Big)
&\leq \Omega_1(f, I)
\leq \Omega_2(f, I)\\
&\leq 2\sum_{i=1}^{n}\langle f(C_i)x,x\rangle,
\end{align*}
where
\begin{align*}
\Omega_1(f, I)&=f\Big(\sum_{i\in I}\frac{p_i}{p_I}\langle C_ix,x\rangle\Big)+f\Big(\frac{p_i}{p_{I^{c}}} \sum_{i\in I^{c}} p_i\langle C_ix,x\rangle\Big)\\
\Omega_2(f, I)&=2\sum_{i\in I} \frac{p_i}{p_I} \langle f(C_i)x,x\rangle+2\sum_{i\in I^{c}}\frac{p_i}{p_{I^{c}}}  \langle f(C_i)x,x\rangle.
\end{align*}
\end{corollary}
\begin{proof}
By rearranging the terms in
$f\Big(\sum_{i=1}^{n}p_i\langle C_ix,x\rangle\Big)$ we reach
\begin{align*}
f\Big(\sum_{i=1}^{n}p_i\langle C_ix,x\rangle\Big)
&=f\Big(p_I(\frac{1}{p_I}\sum_{i\in p_I}p_i\langle C_ix,x\rangle)+p_{I^c}(\frac{1}{p_{I^c}}\sum_{i\in I^c}p_i\langle C_ix,x\rangle)\Big)\\
&\leq f\Big(\frac{1}{p_I} \sum_{i\in I} p_i\langle C_ix,x\rangle\Big)+f\Big(\frac{1}{p_{I^{c}}} \sum_{i\in I^{c}} p_i\langle C_ix,x\rangle\Big)\\
&=\Omega_1(f, I).
\end{align*}
On the other hand, Corollary \ref{M-Mond-pclass-cor} infers
\begin{align*}
\Omega_1(f, I)&=f\Big(\sum_{i\in I} \frac{p_i}{p_I}\langle C_ix,x\rangle\Big)+f\Big(\sum_{i\in I^{c}} \frac{p_i}{p_{I^{c}}}\langle C_ix,x\rangle\Big)\\
&\leq 2\sum_{i\in I}\frac{p_i}{p_I} \langle f(C_i)x,x\rangle+2\sum_{i\in I^{c}}\frac{p_i}{p_{I^{c}}} \langle f(C_i)x,x\rangle)\\
&=\Omega_2(f, I)\\
&\leq 2\sum_{i\in I}\langle f(C_i)x,x\rangle+2\sum_{i\in I^{c}}\langle f(C_i)x,x\rangle\\
&= 2\sum_{i=1}^{n}\langle f(C_i)x,x\rangle.
\end{align*}
\end{proof}

%
\begin{corollary}
Let $f$ be a non-decreasing $P$-class function on $[m,M]$ and let $C_i$, $p_i\geq 0$, $I$, $I^{c}$, $p_{I}$, and $p_{I^{c}}$
be as in Corollary \ref{M-Mond-sharp-pclass}. Then
\begin{align*}
f\Big(\Big|\Big|\sum_{i=1}^{n}p_iC_i\Big|\Big|\Big)
&\leq f\Big(\Big|\Big|\sum_{i\in I}\frac{p_i}{p_I}C_i\Big|\Big|\Big)+f\Big(\Big|\Big|\sum_{i\in I^{c}}\frac{p_i}{p_{I^{c}}} C_i\Big|\Big|\Big)\\
&\leq 2\Big|\Big|\sum_{i\in I}\frac{p_i}{p_I}f(C_i)\Big|\Big|+2\Big|\Big|\sum_{i\in I^{c}}\frac{p_i}{p_{I^{c}}}f(C_i)\Big|\Big|\\
&\leq 2\Big|\Big|\sum_{i=1}^{n}f(C_i)\Big|\Big|.
\end{align*}
\end{corollary}
\begin{proof}
We have
\begin{align*}
f\Big(\Big|\Big|\sum_{i=1}^{n}p_iC_i\Big|\Big|\Big)
&=f\Big(\Big|\Big|\sum_{i\in I}p_iC_i+\sum_{i\in I^{c}}p_iC_i\Big|\Big|\Big)\\
&\leq f\Big(\Big|\Big|\sum_{i\in I}p_iC_i\Big|\Big|+\Big|\Big|\sum_{i\in I^{c}}p_iC_i\Big|\Big|\Big)\\
&= f\Big(p_I\Big|\Big|\frac{1}{p_I}\sum_{i\in I}p_iC_i\Big|\Big|+p_{I^{c}}\Big|\Big|\frac{1}{p_{I^{c}}}\sum_{i\in I^{c}}p_iC_i\Big|\Big|\Big)\\
&\leq f\Big(\frac{1}{p_I}\Big|\Big|\sum_{i\in I}p_iC_i\Big|\Big|\Big)+f\Big(\frac{1}{p_{I^{c}}}\Big|\Big|\sum_{i\in I^{c}}p_iC_i\Big|\Big|\Big).
\end{align*}
On the other hand and by virtue of Corollary \ref{M-Mond-pclass-cor} we get
\begin{align*}
f&\Big(\frac{1}{p_I}\Big|\Big|\sum_{i\in I}p_iC_i\Big|\Big|\Big)+f\Big(\frac{1}{p_{I^{c}}}\Big|\Big|\sum_{i\in I^{c}}p_iC_i\Big|\Big|\Big)\\
&=f\Big(\frac{1}{p_I}\sup_{||x||=1}\langle\sum_{i\in I}p_iC_ix,x\rangle\Big)+f\Big(\frac{1}{p_{I^{c}}}\sup_{||x||=1}\langle\sum_{i\in I^{c}}p_iC_ix,x\rangle\Big)\\
&=\sup_{||x||=1}f\Big(\frac{1}{p_I}\langle\sum_{i\in I}p_iC_ix,x\rangle\Big)+\sup_{||x||=1}f\Big(\frac{1}{p_{I^{c}}}\langle\sum_{i\in I^{c}}p_iC_ix,x\rangle\Big)\\
&\leq 2\sup_{||x||=1}\sum_{i\in I}\frac{p_i}{p_I}\langle f(C_i)x,x\rangle+2\sup_{||x||=1}\sum_{i\in I^{c}}\frac{p_i}{p_{I^{c}}}\langle f(C_i)x,x\rangle\\
&= 2\Big|\Big|\sum_{i\in I}\frac{p_i}{p_I}f(C_i)\Big|\Big|+2\Big|\Big|\sum_{i\in I^{c}}\frac{p_i}{p_{I^{c}}}f(C_i)\Big|\Big|\\
&\leq 2\Big|\Big|\sum_{i\in I}f(C_i)\Big|\Big|+2\Big|\Big|\sum_{i\in I^{c}}f(C_i)\Big|\Big|\\
&= 2\Big|\Big|\sum_{i=1}^{n}f(C_i)\Big|\Big|.
\end{align*}
\end{proof}

%
\begin{remark}
Let $C_i$, $p_i\geq 0$, $I$, $I^{c}$, $p_{I}$, and $p_{I^{c}}$
be as in Corollary \ref{M-Mond-sharp-pclass}. Then

(i) For $0<r<1$, 
\begin{align*}
\Big|\Big|\sum_{i=1}^{n}p_iC_i\Big|\Big|^r
&\leq \Big|\Big|\sum_{i\in I}\frac{p_i}{p_I}C_i\Big|\Big|^r+\Big|\Big|\sum_{i\in I^{c}}\frac{p_i}{p_{I^{c}}} C_i\Big|\Big|^r\\
&\leq 2\Big|\Big|\sum_{i\in I}\frac{p_i}{p_I}C_i^r\Big|\Big|+2\Big|\Big|\sum_{i\in I^{c}}\frac{p_i}{p_{I^{c}}}C_i^r\Big|\Big|\\
&\leq 2\Big|\Big|\sum_{i=1}^{n}C_i^r\Big|\Big|.
\end{align*}

(ii) For $r>1$, and applying part (i) for $\frac{1}{r}<1$ and replacing $C_i^r$ with $C_i$ we conclude
\begin{align*}
\Big|\Big|\sum_{i=1}^{n}p_iC_i^r\Big|\Big|
&\leq \Big(\Big|\Big|\sum_{i\in I}\frac{p_i}{p_I}C_i^r\Big|\Big|^{\frac{1}{r}}+\Big|\Big|\sum_{i\in I^{c}}\frac{p_i}{p_{I^{c}}} C_i^r\Big|\Big|^{\frac{1}{r}}\Big)^r\\
&\leq 2^r\Big(\Big|\Big|\sum_{i\in I}\frac{p_i}{p_I}C_i\Big|\Big|+\Big|\Big|\sum_{i\in I^{c}}\frac{p_i}{p_{I^{c}}}C_i\Big|\Big|\Big)^r\\
&\leq \Big(2\Big|\Big|\sum_{i=1}^{n}C_i\Big|\Big|\Big)^r.
\end{align*}
\end{remark}

%
\begin{theorem}\label{Multi-thm3}
Let the conditions of Theorem \ref{M-Mond-pclass} be satisfied. Then
\begin{equation}
 \sum_{i=1}^{n}\langle f(C_i)x_i, x_i\rangle\leq f(m)+f(M).
\end{equation}
\end{theorem}
\begin{proof}
Consider $\tilde{C}$ and $\tilde{x}$ as in the proof of Theorem \ref{M-Mond-pclass}
and apply Theorem \ref{thm3}.
\end{proof}

%
\begin{theorem}
Let the conditions of Theorem \ref{M-Mond-pclass} be satisfied.
Let $J$ be an interval such that $f([m, M])\subset J$. 
If $F(u, v)$ is a real function defined on $J\times J$ and non--decreasing in $u$, then
\begin{align}
F\Big(2\sum_{i=1}^{n}\langle f(C_i)x_i,x_i\rangle,f\Big(\sum_{i=1}^{n}\langle C_ix_i,x_i\rangle\Big)\Big)
&\leq\max_{t\in[m,M]}F(2(f(m)+f(M)),f(t)).
\end{align}
\end{theorem}
\begin{proof}
Consider $\tilde{C}$ and $\tilde{x}$ as in the proof of Theorem \ref{M-Mond-pclass}
and apply Theorem \ref{thm4}.
\end{proof}

\begin{corollary} 
Let the conditions of Theorem \ref{M-Mond-pclass} be satisfied. Then
\begin{itemize}
\item[(i)] the inequality
\begin{equation}
2\sum_{i=1}^{n}\langle f(C_i)x_i,x_i\rangle\leq \lambda f(\langle \sum_{i=1}^{n}C_ix_i, x_i\rangle)
\end{equation}
holds for some $\lambda>0$,
\item[(ii)] the inequality
\begin{equation}
2\sum_{i=1}^{n}\langle f(C_i)x_i,x_i\rangle\leq \lambda +f(\langle \sum_{i=1}^{n}C_ix_i, x_i\rangle)
\end{equation}
holds for some $\lambda\in\Bbb{R}$.
\end{itemize}
\end{corollary}
\begin{proof}
Consider $\tilde{C}$ and $\tilde{x}$ as in the proof of Theorem \ref{M-Mond-pclass}.

(i) Apply Corollary \ref{cor4} (i) and note that $\lambda=\frac{2(f(m)+f(M))}{\min_{t\in[m,M]}f(t)}$.

(ii) Apply Corollary \ref{cor4} (ii) and note that $\lambda=2(f(m)+f(M))-\min_{t\in[m,M]}f(t)$.
\end{proof}

\bibliographystyle{amsplain}

\end{document}